\documentclass{article}

\usepackage{arxiv}

\usepackage[utf8]{inputenc} 
\usepackage[T1]{fontenc}    
\usepackage{hyperref}       
\usepackage{url}            
\usepackage{booktabs}       
\usepackage{amsfonts}       
\usepackage{nicefrac}       
\usepackage{microtype}      
\usepackage{lipsum}		
\usepackage{graphicx}
\usepackage{mathtools}
\usepackage{amssymb}
\usepackage{amsthm}
\usepackage{amsmath}
\usepackage{natbib}
\usepackage{doi}
\usepackage{mathtools, cases}

\allowdisplaybreaks

\newcommand{\arctg}{\operatorname{arctg}} 

\newcommand{\sgn}{\operatorname{sgn}}
 
\renewcommand{\sinh}{\operatorname{sinh}}
\renewcommand{\cosh}{\operatorname{cosh}}

\newtheorem{thm}{Theorem}

\title{Hitting distribution of a correlated planar Brownian motion in a disk}

\author{ 
	\href{https://orcid.org/0000-0002-6163-044X}{Manfred Marvin Marchione}\\
	Department of Statistical Sciences\\
	Sapienza University of Rome\\
	\texttt{manfredmarvin.marchione@uniroma1.it} \\
	\And
	\href{https://orcid.org/0000-0002-6421-533X}{Enzo Orsingher} \\
	Department of Statistical Sciences\\
	Sapienza University of Rome\\
	\texttt{enzo.orsingher@uniroma1.it}}

\date{December 31, 2021}

\renewcommand{\headeright}{}
\renewcommand{\undertitle}{}

\begin{document}
\maketitle

\begin{abstract}
In this paper we study the hitting probability of a circumference $C_R$ for a correlated Brownian motion $\underline{B}(t)=\left(B_1(t), B_2(t)\right)$, $\rho$ being the correlation coefficient. The analysis starts by first mapping the circle $C_R$ into an ellipse $E$ with semiaxes depending on $\rho$ and transforming the differential operator governing the hitting distribution into the classical Laplace operator. By means of two different approaches (one obtained by applying elliptic coordinates) we obtain the desired distribution as a series of Poisson kernels.
\end{abstract}


\section{Introduction}
The problem of hitting the boundary of a connected set $D$ by a Brownian motion is transformed into the solution of a Dirichlet problem
\begin{equation}\label{generalDirichlet}\begin{dcases}\Delta u=0\qquad\qquad & \normalfont{\text{in}}\;D\\
u=h\qquad\qquad &\normalfont{\text{on}}\;\partial D.\end{dcases}\end{equation}
The probabilistic interpretation of the problem (\ref{generalDirichlet}) is well-known (see, for example, \citeauthor{portstone} \citeyear{portstone}) and the solution has the form \begin{equation}\label{generalDirichletSolution}u(\underline{x})=\mathbb{E}\left(h(\tau_{\partial D})\right)\end{equation}
where $$\tau_{\partial D}=\inf\left\{t>0:\;\underline{B}(t)\in\partial D\right\}$$ and $\underline{B}(t)$ is a Brownian motion starting at $\underline{x}$.\\
The explicit form of (\ref{generalDirichletSolution}) depends on the structure of the set $D$ and is well-known for some specific sets like circles (hyperspheres in $\mathbb{R}^d$), halfplanes, quarters of planes. For sets with a more complicated geometrical structure, an explicit form of (\ref{generalDirichletSolution}), and thus the hitting place distribution of a Brownian motion, becomes difficult to find even in the two-dimensional case.

Most of the results in the literature concern the case of Brownian motions with independent components. Correlated Brownian motions are studied mainly in a financial context. For example, the joint distribution of maxima and minima of two correlated Brownian motions $B_1(t)$ and $B_2(t)$ was studied for pricing exotic options (\citeauthor{chuang} \citeyear{chuang}, \citeauthor{he} \citeyear{he}).\\
In the case of planar Brownian motion with correlated components, the problem (\ref{generalDirichlet}) becomes
\begin{equation}\label{generalCorrelatedDirichlet}\begin{dcases}\frac{\partial^2u}{\partial x^2}+2\rho\frac{\partial^2u}{\partial x\partial y}+\frac{\partial^2u}{\partial y^2}=0\qquad\qquad & (x,y)\in D\\
u(x,y)=h(x,y)\qquad\qquad &(x,y)\in\partial D,\end{dcases}\end{equation}
$\rho$ being the correlation coefficient between the components of the Brownian motion $\underline{B}(t)$. By means of a suitable linear transformation the differential operator appearing in (\ref{generalCorrelatedDirichlet}) can be transformed into the classical Laplace operator. Such a transformation modifies the geometrical structure of $D$, possibly complicating the solution of the Dirichlet problem. \citeauthor{iyengar} (\citeyear{iyengar}) and \citeauthor{metzler} (\citeyear{metzler}) analyzed the distribution of the exit point from a quarter of plane for a correlated planar Brownian motion $\underline{B}(t)$. By mapping the quadrant into a wedge with amplitude depending on $\rho$, they obtained the solution in terms of a series of Bessel functions. A similar approach has also been used in particle diffusion problems.
\citeauthor{majumdar} (\citeyear{majumdar}) mapped a square into a parallelogram in order to obtain the distribution of the maximum distance between the leader and the laggard for three independent Brownian motions in $\mathbb{R}$.\\
This paper deals with the hitting point distribution of a correlated planar Brownian motion in a disk $C_R$ of radius $R$ centered at the origin. We are interested in finding the hitting probability of the boundary $\partial C_R$ as a function of the starting point $(x,y)$
\begin{equation}\label{eq1}P\left\{\underline{B}(\tau_{\partial C_R})\in d\sigma|\underline{B}(0)=(x,y)\right\}=u(x,y)\end{equation} where $d\sigma$ represents an element of $\partial C_R$ and $h(x,y)$ is a Dirac delta function placed at the exit point $d\sigma$. The problem is tackled by transforming the disk $\partial C_R$ into an ellipse with semi-axis depending on the correlation coefficient $\rho$.\\
The Dirichlet problem for the classical Laplace operator in an oval with the form $$D=\left\{(x,y):\;\frac{x^2}{a^2}+\frac{y^2}{b^2}<1\right\}$$
has been studied, from a probabilistic point of view, by \citeauthor{yin} (\citeyear{yin}). The authors obtained the joint distribution of the hitting time and point for an uncorrelated Brownian motion in an ellipse in terms of a double series of Mathieu functions. In the mathematical physics literature, the Green's function for the Dirichlet problem in an ellipse has been obtained in elliptic coordinates (see \citeauthor{morse} (\citeyear{morse}), pag. 1202).\\
In this paper we compare two different approaches we found in the literature: the first one is based on a suitable transformation of the ellipse into a circular annulus as proposed by \citeauthor{ghizzetti} (\citeyear{ghizzetti}), while the second approach is based on the application of elliptic coordinates. We prove the equivalence of the two approaches by showing that elliptic coordinates and Ghizzetti's transformation are stricly related. Furthermore, differently from other authors, we express the distribution of the hitting point $\underline{B}(\tau_{\partial C_R})$ in terms of the superposition of Poisson kernels with different starting points.

\section{Cholesky decomposition: transforming the circle into an ellipse}
In order to solve the Dirichlet problem
\begin{equation}\label{eq3}\begin{dcases}\frac{\partial^2u}{\partial x^2}+2\rho\frac{\partial^2u}{\partial x\partial y}+\frac{\partial^2u}{\partial y^2}=0\qquad\qquad & (x,y)\in C_R\\
u(x,y)=h(x,y)\qquad\qquad &(x,y)\in\partial C_R\end{dcases}\end{equation}
our first step is the transformation of the differential operator in (\ref{eq3}) into the classical bivariate Laplace equation, as shown in the next theorem.

\begin{thm}\label{choleskythm}Consider the linear transformation
\begin{equation}\label{linear}
\begin{pmatrix}
w\\
z
\end{pmatrix}
=\begin{pmatrix}
\frac{1}{\sqrt{2(1-|\rho|)}} & -\frac{\sgn(\rho)}{\sqrt{2(1-|\rho|)}}\\
\frac{\sgn(\rho)}{\sqrt{2(1+|\rho|)}} & \frac{1}{\sqrt{2(1+|\rho|)}}
\end{pmatrix}\cdot
\begin{pmatrix}
x\\
y
\end{pmatrix}.
\end{equation}
The following relationship holds:
$$\frac{\partial^2}{\partial x^2}+2\rho\frac{\partial^2}{\partial x\partial y}+\frac{\partial^2}{\partial y^2}=\frac{\partial^2}{\partial w^2}+\frac{\partial^2}{\partial z^2}.$$
Moreover, the circle $\partial C_R=\{(x,y): x^2+y^2=R^2\}$ is mapped into the canonical ellipse \begin{equation}\label{canonicalellipse}\frac{w^2}{a^2}+\frac{z^2}{b^2}=1\end{equation}
where the semiaxis $a=\frac{R}{\sqrt{1-|\rho|}}$ and $b=\frac{R}{\sqrt{1+|\rho|}}$ are functions of the correlation coefficient $\rho$.
\end{thm}

\begin{proof}We restrict ourselves to outlining the idea behind the proof since the calculations are straightforward. The transformation (\ref{linear}) can be decomposed in the following manner:
$$\begin{pmatrix}w\\z\end{pmatrix}=M\cdot N\cdot\begin{pmatrix}x\\y\end{pmatrix}$$
where $$M=\begin{pmatrix}\sqrt{\frac{1+|\rho|}{2}} & -\sgn(\rho)\sqrt{\frac{1-|\rho|}{2}}\\\sgn(\rho)\sqrt{\frac{1-|\rho|}{2}}&\sqrt{\frac{1+|\rho|}{2}}\end{pmatrix},
\qquad\qquad N=\begin{pmatrix}\frac{1}{\sqrt{1-\rho^2}}&\frac{-\rho}{\sqrt{1-\rho^2}}\\0&1\end{pmatrix}.$$
The matrix $N$ describes a Cholesky decomposition of the covariance matrix of the Brownian motion which transforms $C_R$ into a rotated ellipse and eliminates the correlation from the differential operator as desired. The matrix $M$ corresponds to a counterclockwise rotation of angle $\alpha=\sgn(\rho)\cdot\arctg\sqrt{\frac{1-|\rho|}{1+|\rho|}}$ which reduces the obtained ellipse to the canonical form (\ref{canonicalellipse}).
\end{proof}

We have thus transformed the initial Dirichlet problem (\ref{eq3}) in a circle into the Dirichlet problem in the elliptic set $E=\left\{w^2(1-|\rho|)+z^2(1+|\rho|)<R^2\right\}$
\begin{equation}\label{zwDirich}\begin{dcases}\frac{\partial^2u}{\partial w^2}+\frac{\partial^2u}{\partial z^2}=0\qquad\qquad &(w,z)\in E\\
u(w,z)=h(w,z)\qquad &(w,z)\in\partial E\end{dcases}\end{equation}
where we have maintained, for simplicity, the same symbol for the unknown function. The transformation of (\ref{eq3}) into (\ref{zwDirich}) converts the research of the hitting point of a correlated planar Brownian motion with source inside a disk into the derivation of the distribution of the hitting point for a planar Brownian motion (with independent components) with a starting point inside an elliptic set.

\section{Mapping the ellipse into a circular annulus}
Our next step  is that of transforming the ellipse $E$ into a circular annulus by means of the transformation (\citeauthor{ghizzetti} \citeyear{ghizzetti})
\begin{equation}\label{ghizzetti}\begin{dcases}w=\left(\frac{a+b}{2}r+\frac{a-b}{2}\frac{1}{r}\right)\cos\theta\qquad\qquad0\le\theta\le2\pi\\z=\left(\frac{a+b}{2}r-\frac{a-b}{2}\frac{1}{r}\right)\sin\theta\qquad\qquad q\le r\le1\end{dcases}\end{equation}
where $q=\sqrt{\frac{a-b}{a+b}}$ and $a$ and $b$ are defined as in theorem \ref{choleskythm}.\\
For $r=1$, (\ref{ghizzetti}) maps the ellipse $\partial E$ onto the circumference of radius 1. For $r=q$, the interfocal segment of $E$ is mapped onto the circumference of radius $q$ since, for $r=q$, the transformation (\ref{ghizzetti}) becomes $$\begin{cases}w=\sqrt{a^2-b^2}\cos\theta\qquad\qquad0\le\theta\le2\pi\\z=0.\end{cases}$$
The confocal ellipses of equation $$\dfrac{w^2}{\left(\dfrac{a+b}{2}r+\dfrac{a-b}{2}\dfrac{1}{r}\right)^2}+\dfrac{z^2}{\left(\dfrac{a+b}{2}r-\dfrac{a-b}{2}\dfrac{1}{r}\right)^2}=1$$ are mapped onto circumferences of radius $q<r<1$ (see fig. \ref{fig:figura1} below).

\begin{figure}[h]
\begin{center}
\includegraphics[scale=1]{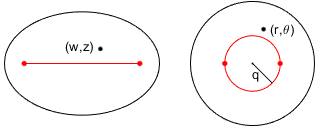}
\end{center}
\vspace{-3mm}

\caption{the transformation (\ref{ghizzetti}) maps the ellipse $E$ ino the annulus of radii $(q,1)$.}
\label{fig:figura1}
\end{figure}

\noindent In order to solve the Dirichlet problem (\ref{zwDirich}) in the $(r,\theta)$ coordinates we use the following relationship:
$$\frac{\partial^2 }{\partial r^2}+\frac{1}{r}\frac{\partial }{\partial r}+\frac{1}{r^2}\frac{\partial^2}{\partial \theta^2}=\left[\left(\frac{a+b}{2}-\frac{a-b}{2}\frac{1}{r^2}\right)^2+\frac{(a^2-b^2)\sin^2\theta}{r^2}\right]\left(\frac{\partial^2}{\partial w^2}+\frac{\partial^2}{\partial z^2}\right).$$

Thus, we solve the problem (\ref{zwDirich}) by setting $f(r,\theta)=u(w,z).$ We start with some remarks. First of all it can be checked that the points $(q,\theta)$ and $(q,-\theta)$ are mapped, by means of the transformation (\ref{ghizzetti}), into the same point in the focal segment of the ellipse $\frac{w^2}{a^2}+\frac{z^2}{b^2}=R^2$. Thus, for the function $u(w,z)$ to be continuous on the focal segment, $f(q,\theta)$ must be an even function of $\theta$, that is $$f(q,\theta)=f(q,-\theta).$$
Analogously, $\frac{\partial}{\partial r}f(r,\theta)|_{r=q}$ must be an odd function of $\theta$, that is \begin{equation}\label{deriv}\frac{\partial}{\partial r}f(r,\theta)|_{r=q}=-\frac{\partial}{\partial r}f(r,-\theta)|_{r=q}.\end{equation}
In fact, the relationship (\ref{ghizzetti}) implies that, for fixed $\theta$, we have the hyperbola of equation\begin{equation}\label{hyperbola}\frac{w^2}{\cos^2\theta}-\frac{z^2}{\sin^2\theta}=a^2-b^2.\end{equation}
The derivative of $f(r,\theta)$ with respect to $r$ can be interpreted as the the slope of the function $f$ along the hyperbola defined by (\ref{hyperbola}). By changing the sign of $\theta$ we obtain the slope of $f$ along the same hyperbola in the opposite orientation. 
It follows that the condition (\ref{deriv}) must hold.

\begin{figure}[h]
\begin{center}
\includegraphics[scale=1]{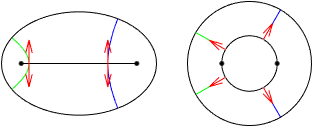}
\end{center}
\vspace{-3mm}

\caption{the derivative $\frac{\partial}{\partial r} f(r,\theta)$ represents, for fixed $\theta$, the slope of $f$ along the branch of hyperbola identified by $\theta$. Thus, if the derivative is evaluated in opposite orientations (red arrows) along the hyperbola, the sign must change.}
\end{figure}

\noindent We can finally reformulate problem (\ref{zwDirich}) in the following manner:

\begingroup
\addtolength{\jot}{0.5em}
\begin{subequations}\label{rthetaDirich:test}
\begin{numcases}{}  
\frac{\partial^2 f}{\partial r^2}+\frac{1}{r}\frac{\partial f}{\partial r}+\frac{1}{r^2}\frac{\partial^2 f}{\partial \theta^2}=0\qquad\qquad & $r<1$\label{rthetaDirich:eq}\\
f(r,\theta)=h(\theta) & $r=1$\label{rthetaDirich:bound}\\
f(q,\theta)\;\text{is even}\label{rthetaDirich:c1}\\
\frac{\partial}{\partial r}f(r,\theta)|_{r=q}\;\text{is odd}.\label{rthetaDirich:c2}
\end{numcases}
\end{subequations}
\endgroup

The reader is warned not to confuse the problem (\ref{rthetaDirich:test}) with the classical Dirichlet problem for an uncorrelated Brownian motion in a circular annulus. The distribution of the hitting place of a Brownian motion starting from inside a concentric sperical shell in $\mathbb{R}^d$ is well-known and \citeauthor{wendel} (\citeyear{wendel}) has also obtained the joint distribution of the hitting time and hitting point. As for the problem analyzed in this paper, mapping the ellipse to a circular annulus modifies the nature of the examined stochastic process whose transition function satisfies, after the transformation, the following partial differential equation:
$$\frac{\partial u}{\partial t}=\frac{1}{\left(\frac{a+b}{2}-\frac{a-b}{2}\frac{1}{r^2}\right)^2+\frac{(a^2-b^2)\sin^2\theta}{r^2}}\left(\frac{\partial^2 }{\partial r^2}+\frac{1}{r}\frac{\partial }{\partial r}+\frac{1}{r^2}\frac{\partial^2}{\partial \theta^2}\right)u.$$
Observe that the radial coordinate $r$ of the transformed stochastic process is not a Bessel process as it would be in the classical hitting point problem in a circular annulus. Moreover, the Brownian paths which cross the focal segment in the ellipse are mapped to discontinuous paths in the circular annulus, as shown in figure \ref{discontinuous}.\\
\vspace{-4mm}

\begin{figure}[h]
\begin{center}
\includegraphics[scale=1]{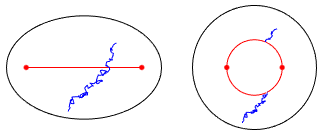}
\end{center}
\vspace{-3mm}

\caption{the transformation (\ref{ghizzetti}) maps the Browian paths crossing the focal segment in $E$ into discontinuous paths in the circular annulus.}
\label{discontinuous}
\end{figure}
\noindent The solution to problem (\ref{rthetaDirich:test}) is given in the next theorem.

\begin{thm}\label{ghizzettithm}The general solution of the Dirichlet problem (\ref{rthetaDirich:test}) is
\begin{align}\label{ghizzettithmsolution}f(r,\theta)=\int_0^{2\pi}h(\tau)\left[\frac{1}{2\pi}+\frac{1}{\pi}\sum_{k=1}^{+\infty}r^k\left(\frac{1+\left(\frac{q}{r}\right)^{2k}}{1+q^{2k}}\cos k\theta\cos k\tau\;\right.\right.\nonumber\\
\left.\left.+\;\frac{1-\left(\frac{q}{r}\right)^{2k}}{1-q^{2k}}\sin k\theta\sin k\tau\right)\right]d\tau.\end{align}
\end{thm}
\begin{proof}The problem (\ref{rthetaDirich:test}) can be solved by applying the method of separation of variables.\\
\noindent By assuming \begin{equation}f(r,\theta)=R(r)\Theta(\theta)\end{equation}we obtain the ordinary differential equations
\begin{align}\label{thetaeq}\Theta''(\theta)+\gamma^2\Theta(\theta)=0\\\label{req}r^2R''(r)+rR'(r)-\gamma^2 R(r)=0\end{align}where $\gamma$ is an arbitrary constant.
The general solutions of (\ref{thetaeq}) and (\ref{req}) are
$$\Theta(\theta)=A\cos\gamma\theta+B\sin\gamma\theta$$
and
$$R(r)=Cr^{\gamma}+Dr^{-\gamma}.$$
For the periodic character of the solution it must be $\gamma=k,\;k\in\mathbb{N}$ and the general solution of equation (\ref{rthetaDirich:eq}) is
\begin{align*}f(r,\theta)=&\frac{A_0}{2}+\sum_{k=1}^{+\infty}\left(a_k\cos k\theta+b_k\sin k\theta\right)\left(c_kr^k+d_kr^{-k}\right) \qquad\qquad\qquad\\=&\frac{A_0}{2}+\sum_{k=1}^{+\infty}\left[A_kr^k\cos k\theta+B_kr^k\sin k\theta+C_kr^{-k}\cos k\theta+D_kr^{-k}\sin k\theta\right].\end{align*}
In force of condition (\ref{rthetaDirich:c1}), that is $f(q,\theta)=f(q,-\theta)$, we have that $$B_kq^k=-D_kq^{-k}$$
and in light of (\ref{rthetaDirich:c2}), that is $\frac{\partial}{\partial r}f(r,\theta)|_{r=q}=-\frac{\partial}{\partial r}f(r,-\theta)|_{r=q}$, we obtain $$A_kq^k=C_kq^{-k}.$$
In conclusion, the solution of equation (\ref{rthetaDirich:eq}) with the constraints (\ref{rthetaDirich:c1}) and (\ref{rthetaDirich:c2}) becomes
\begin{align}\label{quasisol}f(r,\theta)=\frac{A_0}{2}+\sum_{k=1}^{+\infty}\Big[&A_k\bigg(r^k+\frac{q^{2k}}{r^k}\bigg)\cos k\theta\;\\
+\;&B_k\bigg(r^k-\frac{q^{2k}}{r^k}\bigg)\sin k\theta\Big].\nonumber\end{align}
For $r=1$ we must have
\begin{align}\label{quasisolboundary}f(1,\theta)=h(\theta)=\frac{A_0}{2}+\sum_{k=1}^{+\infty}\big[&A_k(1+q^{2k})\cos k\theta\;\\
+\;&B_k(1-q^{2k})\sin k\theta\big].\nonumber\end{align}
\noindent The Fourier coefficients of (\ref{quasisolboundary}) are therefore \begin{equation}\label{fourier}\begin{dcases}A_k=\frac{1}{\pi}\frac{1}{(1+q^{2k})}\int_0^{2\pi}h(\tau)\cos k\tau d\tau\\ B_k=\frac{1}{\pi}\frac{1}{(1-q^{2k})}\int_0^{2\pi}h(\tau)\sin k\tau d\tau\end{dcases}\end{equation}
\noindent The substitution of the Fourier coefficients (\ref{fourier}) into equation (\ref{quasisol}) completes the proof.
\end{proof}

The solution obtained in theorem \ref{ghizzettithm} can be expressed in terms of series of Poisson kernels. Indeed, the kernel of formula (\ref{ghizzettithmsolution}) can be written as
\begin{align}&\frac{1}{2\pi}+\frac{1}{\pi}\sum_{k=1}^{+\infty}r^k\left(\frac{1+\left(\frac{q}{r}\right)^{2k}}{1+q^{2k}}\cos k\theta\cos k\tau+\;\frac{1-\left(\frac{q}{r}\right)^{2k}}{1-q^{2k}}\sin k\theta\sin k\tau\right)\nonumber\\
&=\frac{1}{2\pi }+\frac{1}{\pi }\sum ^{\infty }_{k=1}r^{k}\cos \left( k\left( \theta -\tau \right) \right)+\frac{1}{\pi}\sum_{k=1}^{\infty}r^k\Bigg[\frac{q^{2k}\left(\frac{1}{r^{2k}}-1\right)}{1+q^{2k}}\cos k\theta \cos k\tau\nonumber\\
&\qquad\qquad\qquad\qquad\qquad\qquad\qquad\qquad\qquad-\frac{q^{2k}\left(\frac{1}{r^{2k}}-1\right)}{1-q^{2k}}\sin k\theta \sin k\tau\Bigg]\nonumber\\
&=\frac{1}{2\pi }+\frac{1}{\pi }\sum ^{\infty }_{k=1}r^{k}\cos \left( k\left( \theta -\tau \right) \right)+\frac{1}{\pi}\sum_{k=1}^{\infty}\frac{\dfrac{1}{r^k}-r^k}{1-q^{4k}}\Bigg[q^{2k}\cos\left(k\left(\theta+\tau\right)\right)\nonumber\\
&\qquad\qquad\qquad\qquad\qquad\qquad\qquad\qquad\qquad\qquad\;-q^{4k}\cos\left(k\left(\theta-\tau\right)\right)\Bigg]\nonumber\\
&=\frac{1}{2\pi }+\frac{1}{\pi }\sum ^{\infty }_{k=1}r^{k}\cos \left( k\left( \theta -\tau \right) \right)+\frac{1}{\pi}\sum_{j=0}^{\infty}\Bigg[\sum_{k=1}^{\infty}\left(\frac{q^{2+4j}}{r}\right)^k\cos\left(k\left(\theta+\tau\right)\right)\nonumber\\
&\qquad\qquad-\sum_{k=1}^{\infty}\left(rq^{2+4j}\right)^k\cos\left(k\left(\theta+\tau\right)\right)-\sum_{k=1}^{\infty}\left(\frac{q^{4+4j}}{r}\right)^k\cos\left(k\left(\theta-\tau\right)\right)\nonumber\\
&\qquad\qquad+\sum_{k=1}^{\infty}\left(rq^{4+4j}\right)^k\cos\left(k\left(\theta-\tau\right)\right)\Bigg]=\nonumber\\
&=\frac{1}{2\pi }\;\frac{1-r^2}{1-2r\cos(\theta-\tau)+r^2}+\frac{1}{2\pi}\sum_{j=0}^{\infty}\Bigg[\frac{r^2-q^{4(2j+1)}}{r^2-2rq^{2(2j+1)}\cos(\theta+\tau)+q^{4(2j+1)}}\qquad\qquad\qquad\qquad\nonumber\\
&\;\;\;-\frac{1-r^2q^{4(2j+1)}}{1-2rq^{2(2j+1)}\cos(\theta+\tau)+r^2q^{4(2j+1)}}-\frac{r^2-q^{8(j+1)}}{r^2-2rq^{4(j+1)}\cos(\theta-\tau)+q^{8(j+1)}}\nonumber\\
&\;\;\;+\frac{1-r^2q^{8(j+1)}}{1-2rq^{4(j+1)}\cos(\theta-\tau)+r^2q^{8(j+1)}}\Bigg].\label{kernelseries} \end{align}
It is well-known that the classical Poisson kernel for a planar Brownian motion with independent components writes
\begin{equation}\label{classicalPoissonkernel}P\{\underline{B}(T_{\partial C_R})\in d(R,\tau)|\underline{B}(0)=(r,\theta)\}=\frac{1}{2\pi}\frac{R^2-r^2}{R^2-2rR\cos(\theta-\tau)+r^2}\;d\tau.\end{equation}
Other probabilistic interpretations of (\ref{classicalPoissonkernel}) have been proposed in the literature. \citeauthor{orsingher} (\citeyear{orsingher}) obtained Poisson kernels by wrapping up time-changed (with a stable subordinator) pseudoprocesses on a circle.\\
Each term of the sum in (\ref{kernelseries}) has the form (\ref{classicalPoissonkernel}) with different starting points all located inside the circumference of radius $q$. 
Thus, the probability (\ref{eq1}) in the case of correlated Brownian motions can be obtained as the combination of Poisson probabilities with different starting points. When $j$ increases, the starting points approach the center of the disk and  the higher-order terms of the sum become negligible.\\
We have observed that the transformation (\ref{ghizzetti}) maps the Brownian motion moving in the ellipse $E$ into a new process moving inside the circular annulus of radii $(q,1)$. In principle, we should relate the final coordinates $(r,\theta)$ to the initial ones $(x,y)$ and interpret the result (\ref{kernelseries}) in terms of the distribution of $\underline{B}(\tau_{\partial C_R})$ in the original coordinates. Thus, it is necessary to invert the changes of coordinates performed so far. In particular, the inversion of the transformation (\ref{ghizzetti}) is tedious and will be treated in the last section.

\section{Solution in elliptic coordinates}
An alternative way to solve the Dirichlet problem (\ref{zwDirich}) is to use elliptic coordinates, defined by the formulas
\begin{equation}\label{ellipticcoordinates}\begin{dcases}w=\sqrt{a^2-b^2}\cosh\eta\cos\varphi\\z=\sqrt{a^2-b^2}\sinh \eta\sin\varphi.\end{dcases}\end{equation}
\noindent Notice that the ellipse $\partial E$ is obtained for $\eta=\hat{\eta}:=-\log q$ while the focal segment is obtained for $\eta=0$.\\
\noindent It is well known that the laplacian in elliptic coordinates (\citeauthor{lebedev} \citeyear{lebedev}, pag. 204) is
$$\frac{\partial^2}{\partial w^2}+\frac{\partial^2}{\partial z^2}=\frac{1}{(a^2-b^2)\left(\sin^2\varphi+\sinh^2\eta\right)}\left(\frac{\partial^2}{\partial \eta^2}+\frac{\partial^2}{\partial \varphi^2}\right).$$
Thus, by setting $g(\eta,\varphi)=u(w,z)$, we can reformulate the problem (\ref{zwDirich}) as follows:

\begingroup
\addtolength{\jot}{0.5em}
\begin{subequations}\label{etaphiDirich:test}
\begin{numcases}{}  
\frac{\partial^2g}{\partial \eta^2}+\frac{\partial^2g}{\partial \varphi^2}=0\qquad\qquad & $\eta<\hat{\eta}$\label{etaphiDirich:eq}\\
g(\eta,\theta)=h(\varphi) & $\eta=\hat{\eta}$\label{etaphiDirich:bound}\\
g(0,\varphi)\;\text{is even}\label{etaphiDirich:c1}\\
\frac{\partial}{\partial \eta}g(\eta,\varphi)|_{\eta=0}\;\text{is odd}.\label{etaphiDirich:c2}
\end{numcases}
\end{subequations}
\endgroup
where conditions (\ref{etaphiDirich:c1}) and (\ref{etaphiDirich:c2}) can be justified as in the previous section.\\
The general solution to the problem (\ref{etaphiDirich:eq}) is given in the next theorem.

\begin{thm}\label{ellipticthm}The general solution of the Dirichlet problem (\ref{etaphiDirich:test}) is
\begin{align}g(\eta,\varphi)=\int_0^{2\pi}h(\tau)\left[\frac{1}{2\pi}+\frac{1}{\pi}\sum_{k=1}^{+\infty}\frac{\sinh\eta}{\sinh\hat{\eta}}\cos k\varphi\cos k\tau\;\right.\nonumber\\
\left.+\;\frac{\cosh\eta}{\cosh\hat{\eta}}\sin k\varphi\sin k\tau\right]d\tau.\end{align}
\end{thm}
\begin{proof}We apply again the method of separation of variables as in theorem \ref{ghizzettithm}. By assuming \begin{equation}g(\eta,\varphi)=H(\eta)\Phi(\varphi)\end{equation}we obtain the ordinary equations
\begin{align}\label{varphieq}\Phi''(\varphi)+\gamma^2\Phi(\varphi)=0\\\label{etaeq}H''(\eta)-\gamma^2H(\eta)=0\end{align}where $\gamma$ is an arbitrary constant.
The general solutions of (\ref{varphieq}) and (\ref{etaeq}) are
$$\Phi(\varphi)=A\cos\gamma\varphi+B\sin\gamma\varphi$$
and
$$H(\eta)=C\cosh\gamma\eta+D\sinh\gamma\eta.$$
For the function $g(\eta,\varphi)$ to be periodic in $\varphi$ it must be $\gamma=k,\;k\in\mathbb{N}$ and the general solution of equation (\ref{etaphiDirich:eq}) is
\begin{align*}g(\eta,\varphi)=&\frac{A_0}{2}+\sum_{k=1}^{+\infty}\left(a_k\cos k\varphi+b_k\sin k\varphi\right)\left(c_k\cosh k\eta+d_k\sinh k\eta\right) \\=&\frac{A_0}{2}+\sum_{k=1}^{+\infty}\bigg[A_k\cosh k\eta\cos k\varphi+B_k\cosh k\eta\sin k\varphi\\
&\qquad\qquad+\;C_k\sinh k\eta\cos k\varphi+D_k\sinh k\eta\sin k\varphi\bigg].\end{align*}
In force of conditions (\ref{etaphiDirich:c1}) and (\ref{etaphiDirich:c2}) we have that $$B_k=C_k=0.$$
Thus, the general solution of equation (\ref{etaphiDirich:eq}) with the constraints (\ref{etaphiDirich:c1}) and (\ref{etaphiDirich:c2}) is
\begin{align}\label{ellipticquasisol}g(\eta,\varphi)=\frac{A_0}{2}+\sum_{k=1}^{+\infty}\big[A_k\cosh k\eta\cos k\varphi+D_k\sinh k\eta\sin k\varphi\big].\end{align}
For $\eta=\hat{\eta}$ we must have
\begin{align}\label{ellipticquasisolboundary}g(\hat{\eta},\varphi)=h(\varphi)=\frac{A_0}{2}+\sum_{k=1}^{+\infty}\big[A_k\cosh k\hat{\eta}\cos k\varphi+D_k\sinh k\hat{\eta}\sin k\varphi\big]\end{align}
\noindent The Fourier coefficients of (\ref{ellipticquasisolboundary}) are \begin{equation}\label{ellipticfourier}
\begin{dcases}A_k=\frac{1}{\pi\cosh k\hat{\eta}}\int_0^{2\pi}h(\tau)\cos k\tau d\tau\\ D_k=\frac{1}{\pi\sinh k\hat{\eta}}\int_0^{2\pi}h(\tau)\sin k\tau d\tau\end{dcases}\end{equation}
\noindent Plugging formulas (\ref{ellipticfourier}) into (\ref{ellipticquasisol}) completes the proof.
\end{proof}

\section{Inverting the changes of coordinates}
We now treat the problem of inverting the changes of coordinates performed so far in order to express the distribution of $\underline{B}(T_{\partial C_R})$ in terms of the original coordinates $(x,y)$. We restrict ourselves to outlining the inversion of the non-linear transformation (\ref{ghizzetti}) and disregard the obvious inversion of the initial linear transformation (\ref{linear}).
By setting
\begin{equation}\label{ABdef}A=\frac{a+b}{2},\qquad B=\frac{a-b}{2}\end{equation}
\noindent we can write formulas (\ref{ghizzetti}) in the form
\begin{equation}\label{wdefAB}w=\left(Ar+B\frac{1}{r}\right)\cos\theta\end{equation}
\begin{equation}\label{zdefAB}z=\left( Ar-B\frac{1}{r}\right)\sin\theta.\end{equation}
\noindent By summing the squares of (\ref{wdefAB}) and (\ref{zdefAB}) we obtain
\begin{equation}\label{somma}w^2+z^2=A^2r^2+B^2\frac{1}{r^2}+2AB(\cos^2\theta-\sin^2\theta).\end{equation}
Taking the difference of the squares yields
\begin{equation}\label{differenza}w^2-z^2=\left(A^2r^2+B^2\frac{1}{r^2}\right)(\cos^2\theta-\sin^2\theta)+2AB.\end{equation}
\noindent Writing formula (\ref{somma}) in the form $$\cos^2\theta-\sin^2\theta=\frac{w^2+z^2-A^2r^2-B^2\frac{1}{r^2}}{2AB}$$ and substituting it into (\ref{differenza}) yields
\begin{equation}\label{preeq1}w^2-z^2=\frac{\left(A^2r^2+B^2\frac{1}{r^2}\right)\left(w^2+z^2-A^2r^2-B^2\frac{1}{r^2}\right)}{2AB}+2AB.\end{equation}
Now we set \begin{equation}\label{sostituzione}m=A^2r^2+B^2\frac{1}{r^2}\end{equation}
which allows us to write equation (\ref{preeq1}) in the form of a second order equation in $m$
$$m^2-(w^2+z^2)m+2AB\left(w^2-z^2-2AB\right).$$
By considering the non-negative root we obtain \begin{equation}\label{msolution}m=\frac{w^2+z^2+\sqrt{(w^2+z^2)^2-2(a^2-b^2)(w^2-z^2)+(a^2-b^2)^2}}{2}.\end{equation}
Equation (\ref{sostituzione}) can now be written as $$A^2r^4-mr^2+B^2=0$$
which implies \begin{equation}\label{whichroot}r^2=\frac{m\pm\sqrt{m^2-4A^2B^2}}{2A^2}.\end{equation}
Recalling that $r\ge q$ for $(w,z)\in E$, we must take the root with positive sign in formula (\ref{whichroot}). Indeed, consider a point  of coordinates $(0,z)\in E$. By taking the negative sign in formula (\ref{whichroot}) we would obtain
$$r^2=\frac{2z^2+(a^2-b^2)-2\sqrt{z^4+z^2(a^2-b^2)}}{(a+b)^2}\le\frac{2z^2+(a^2-b^2)-2\sqrt{z^4}}{(a+b)^2}=q^2$$
which violates the required condition. We have finally obtained an explicit expression of $r$ in terms or $w$ and $z$
\begin{align}\label{rfinal}r=&\sqrt{\frac{m+\sqrt{m^2-4A^2B^2}}{2A^2}}\\=&\frac{\sqrt{w^2+z^2-(a^2-b^2)+\sqrt{(w^2+z^2)^2-2(a^2-b^2)(w^2-z^2)+(a^2-b^2)^2}}}{\sqrt{2}(a+b)}\nonumber\\
&+\frac{\sqrt{w^2+z^2+(a^2-b^2)+\sqrt{(w^2+z^2)^2-2(a^2-b^2)(w^2-z^2)+(a^2-b^2)^2}}}{\sqrt{2}(a+b)}.\nonumber\end{align}
A similar procedure for $\theta$ yields
\begin{equation}\label{thetafinal}\sin\theta=\sqrt{\frac{-(w^2+z^2)+(a^2-b^2)+\sqrt{(w^2+z^2-a^2+b^2)^2+4(a^2-b^2)z^2}}{2(a^2-b^2)}}.\end{equation}
Care must be paid in selecting the correct branch of the $\arcsin$ function while inverting formula (\ref{thetafinal}).\\
We will now obtain an analogous result for the elliptic coordinates defined in (\ref{ellipticcoordinates}). We start by writing the first formula of (\ref{ellipticcoordinates}) in the form
\begin{equation}\label{eci1}\sin^2\varphi=1-\frac{w^2}{(a^2-b^2)\cosh^2 \eta}.\end{equation}
By substituting formula (\ref{eci1}) into the square of the second formula in (\ref{ellipticcoordinates}), we obtain the equation 
$$(a^2-b^2)\sinh^4\eta-\left(w^2+z^2-(a^2-b^2)\right)\sinh^2\eta-z^2=0.$$
By taking the non-negative root, we have that \begin{equation}\label{inverseelliptic}\sinh^2\eta=\frac{w^2+z^2-(a^2-b^2)+\sqrt{(w^2+z^2)^2-2(a^2-b^2)(w^2-z^2)+(a^2-b^2)^2}}{2(a^2-b^2)}.\end{equation}
It can easily be verified that
 \begin{equation}\label{eci2}m=(a^2-b^2)(\sinh^2\eta+\frac{1}{2}).\end{equation}
Thus, by sustituting (\ref{eci2}) into the square of (\ref{rfinal}) and recalling formulas (\ref{ABdef}), we have

\begin{align}r^2&=\frac{2(a^2-b^2)\left(\sinh^2\eta+\frac{1}{2}\right)+\sqrt{4(a^2-b^2)^2\left(\sinh^4\eta+\sinh^2\eta+\frac{1}{4}\right)-(a^2-b^2)^2}}{(a+b)^2}\qquad\qquad\qquad\nonumber\\
&=\frac{2(a^2-b^2)\left(\sinh^2\eta+\frac{1}{2}\right)+2(a^2-b^2)\sqrt{\sinh^2\eta\left(\sinh^2\eta+1\right)}}{(a+b)^2}\qquad\qquad\qquad\nonumber\\
&=2\frac{(a^2-b^2)}{(a+b)^2}\left(\sinh^2\eta+\sinh \eta\cosh \eta+\frac{1}{2}\right)=q^2e^{2\eta}\nonumber\end{align}
\noindent where $q=\sqrt{\frac{a-b}{a+b}}$ as in the previous sections.
\smallskip

We have finally obtained the relationship between the coordinates $(r,\theta)$ on the circular annulus and the elliptic coordinates $(\eta,\varphi)$. Indeed, the substitution of the expression $r=qe^{\eta}$ in (\ref{ghizzetti}) and comparison with (\ref{ellipticcoordinates}) yield
\begin{equation}\label{equivalenceDef}\begin{cases}r=qe^{\eta}\\ \theta=\varphi.\end{cases}\end{equation}
\smallskip

 By means of transformation (\ref{equivalenceDef}), the equivalence of the solutions obtained in theorems \ref{ghizzettithm} and \ref{ellipticthm} is checked.

\bibliographystyle{unsrtnat}

\bibliography{MarchioneOrsingherManuscript} 
\end{document}